\documentclass{article}

\usepackage{arxiv}
\usepackage[utf8]{inputenc} % allow utf-8 input
\usepackage[T1]{fontenc}    % use 8-bit T1 fonts
\usepackage{url}            % simple URL typesetting
\usepackage{booktabs}       % professional-quality tables
\usepackage{amsfonts}       % blackboard math symbols

\usepackage{nicefrac}       % compact symbols for 1/2, etc.
\usepackage{microtype}      % microtypography
\usepackage{bm}
\usepackage{amsmath}
\usepackage{amssymb}
\usepackage{amsthm}
\usepackage{xcolor}
\usepackage{mathtools}
\usepackage{relsize}
\usepackage{xspace}
\usepackage{subfig}
\usepackage{url}
\usepackage[algo2e,ruled,vlined,linesnumbered]{algorithm2e}
\SetKwComment{Comment}{$\triangleright$\ }{}

\usepackage[english]{babel}
\usepackage[utf8]{inputenc}
\usepackage{algorithm}
\usepackage{graphicx}
\usepackage{tikz}

\usepackage[colorlinks=true, allcolors=blue]{hyperref}

\DeclareMathSymbol{\shortminus}{\mathbin}{AMSa}{"39}

\newcommand{\npoints}{n}  % previously we used $\mu$
\newcommand{\BF}[1]{
	\relax
	\ifmmode
	\ifcat\noexpand#1\relax % check if the argument is a control sequence
		\boldsymbol{#1}     % probably Greek
	\else
		\mathbf{#1}
	\fi
	\else
		\textbf{#1}
	\fi
}

\DeclareRobustCommand\iff{\;\Longleftrightarrow\;}
\DeclareMathOperator{\sign}{sgn}
\DeclareMathOperator{\Hom}{Hom} % homomorphism, linear maps

\newcommand{\Dom}[2][m]{\ensuremath{\operatorname{Dom}_{#1}(#2)}}
\newcommand{\HV}{\ensuremath{\operatorname{HV}}\xspace}
\newcommand{\HVY}{\ensuremath{\mathcal{H}}\xspace}
\newcommand{\HVF}{\ensuremath{\mathcal{H}_{\BF{F}}}\xspace}
\newcommand{\HVC}{\ensuremath{\operatorname{HVC}}\xspace}

\newcommand{\domain}{\ensuremath{\mathcal{X}}\xspace} % decision space
\newcommand{\R}{\ensuremath{\mathbb{R}}\xspace} 
\newcommand{\X}{\ensuremath{\mathbf{X}}\xspace} % approximation set
\newcommand{\f}{\ensuremath{\mathbf{f}}\xspace} % objective function
\newcommand{\F}{\ensuremath{\mathbf{F}}\xspace} % objective function of the set-oriented problem
\newcommand{\Y}{\ensuremath{\mathbf{Y}}\xspace} % approximation set to the Pareto front
\newcommand{\proj}{\ensuremath{\operatorname{proj}}\xspace}
\renewcommand{\r}{\ensuremath{\mathbf{r}}\xspace}

\newtheorem{theorem}{Theorem}

\newtheorem{corollary}{Corollary}[theorem]
\newtheorem{lemma}{Lemma}

\title{The Hypervolume Indicator Hessian Matrix: Analytical Expression, Computational Time Complexity, and Sparsity}
\author{Andr\'e H. Deutz\\Leiden Institute of Advanced Computer Science\\Leiden University\\2333CA Leiden\\The Netherlands \And Michael T.M. Emmerich\\Leiden Institute of Advanced Computer Science\\Leiden University\\2333CA Leiden\\The Netherlands \And Hao Wang\\Leiden Institute of Advanced Computer Science\\and applied Quantum algorithms (aQa) \\Leiden University\\2333CA Leiden\\The Netherlands}

\begin{document}
\maketitle

\begin{abstract}
The problem of approximating the Pareto front of a multiobjective optimization problem can be reformulated as the problem of finding a set that maximizes the hypervolume indicator. 
This paper establishes the analytical expression of the Hessian matrix of the mapping from a (fixed size) collection of $n$ points in the $d$-dimensional decision space (or $m$ dimensional objective space) to the scalar hypervolume indicator value. To define the Hessian matrix, the input set is vectorized, and the matrix is derived by analytical differentiation of the mapping from a vectorized set to the hypervolume indicator.
The Hessian matrix plays a crucial role in second-order methods, such as the Newton-Raphson optimization method, and it can be used for the verification of local optimal sets. So far, the full analytical expression was only established and analyzed for the relatively simple bi-objective case. This paper will derive the full expression for arbitrary dimensions ($m\geq2$ objective functions). For the practically important three-dimensional case, we also provide an asymptotically efficient algorithm with time complexity in $O(n\log n)$ for the exact computation of the Hessian Matrix' non-zero entries. We establish a sharp bound of $12m-6$ for the number of non-zero entries. Also, for the general $m$-dimensional case, a compact recursive analytical expression is established, and its algorithmic implementation is discussed. Also, for the general case, some sparsity results can be established; these results are implied by the recursive expression.
To validate and illustrate the analytically derived algorithms and results, we provide a few numerical examples using Python and Mathematica implementations. Open-source implementations of the algorithms and testing data are made available as a supplement to this paper. 
\end{abstract}

\section{Introduction}\label{sec:intro}
%%%%%
\begin{table}[t]
\centering
\begin{tabular}{|l|l|l|}
\hline
Symbol & Domain or Signature & Description\\
\hline
 $m$& $\mathbb{N}$  & number of objective functions   \\
 $d$ & $\mathbb{N}$  & number of decision variables \\
 $\npoints$ & $\mathbb{N}$  & number of points in the approximation set   \\
 $\mathbf{f} = (f_1,\dots, f_m)$ & $\mathbb{R}^d  \rightarrow \mathbb{R}^m$ & vector-valued objective function  \\
 $\mathbf{X} = ({\BF{x}^{(1)}}^\top,\ldots, {\BF{x}^{(\npoints)}}^\top)^\top$  & $\mathbb{R}^{\npoints  d}$  &  concatenation of $\npoints$ points in the decision space\\
 $\mathbf{Y} = ({\BF{y}^{(1)}}^\top,\ldots, {\BF{y}^{(\npoints)}}^\top)^\top$  &  $\R^{\npoints m}$ & concatenation of $\npoints$ points in the objective space\\
 $\lambda_m$ & $\mathcal{B}(\mathbb{R}^{\npoints})\rightarrow \R_{\geq 0}$ & $m$-dimensional Lebesgue measure \\
 $\HV$ & $\mathcal{B}(\mathbb{R}^{\npoints}) \rightarrow \mathbb{R}_{\geq 0}$ & HVI for subsets of the objective space\\
 \HVY & $\R^{\npoints m} \rightarrow \R_{\geq 0}$ & HVI supported on the product of $\npoints$ objective spaces\\
 \HVF & $\R^{\npoints d} \rightarrow \R_{\geq 0}$ & HVI supported on the product of $\npoints$ decision spaces\\
\hline
\end{tabular}
\caption{Basic notation used throughout the paper. HVI stands for ``Hypervolume Indicator''. $\mathcal{B}(\mathbb{R}^{\npoints})$ denotes the Borel sets on $\R^n$.}
\end{table}

In this paper, we delve into continuous $m$-dimensional multi-objective optimization problems (MOPs), where multiple objective functions, e.g., $\f=(f_1, \ldots, f_m): \domain\subseteq \R^d \rightarrow \R^m$ are subject to minimization. Also, we assume $\f$ is at least twice continuously differentiable. When solving such problems, it is a common strategy to approximate the Pareto front for $m$-objective functions mapping from a continuous decision space $\R^d$ to the $\R$ (or as a vector-valued function from $\R^d$ to $\R^m$. MOPs can be accomplished by means of a finite set of points that distributes across the at most $m-1$-dimensional manifold of the Pareto front. The hypervolume indicator of a set of points is the $m$ dimensional Lebesgue measure of the space that is jointly dominated by a set of objective function vectors in $\R^m$ and bound from above by a reference point $\r\in \R^m$.
More precisely, for minimization problems, the hypervolume indicator (HV)~\cite{ZitzlerT98,ZitzlerTLFF03} is defined as the Lebesgue measure of the compact set dominated by a Pareto approximation set $Y\subset \R^m$ and cut from above by a reference point $\BF{r}$:
%%%%
$$
 \operatorname{HV}(Y;\BF{r}) = \lambda_m\left(\left\{\BF{p}\colon \exists\BF{y}\in Y(\BF{y}\prec\BF{p} \wedge \BF{p}\prec \BF{r})\right\}\right),
$$
%%%%
where $\lambda_m$ denotes the Lebesgue measure on measurable space $(\R^m, B(\R^m))$ with $B(\R^m)$ being the Borel sets of $\R^m$. We will omit the reference point for simplicity henceforth. HV is Pareto compliant, i.e., for all $Y \prec Y'$, $\HV(Y) > \HV(Y')$, and is extensively used to assess the quality of approximation sets to the Pareto front, e.g., in SMS-EMOA~\cite{BeumeNE07} and multiobjective Bayesian optimization~\cite{EmmerichYD0F16}. Being a set function, it is cumbersome to define the derivative of HV\footnote{The derivative of a set function is not defined for an arbitrary family of sets. For some special cases, it can be defined directly, e.g., on Jordan-measurable sets~\cite{dibenedetto2002real}.}.
Therefore, we follow the generic set-based approach for MOPs~\cite{EmmerichD12}, which considers a finite set of objective points (of size $\npoints$) as a single point in $\R^{\npoints m}$, obtained via the following concatenation map (and its inverse):
%%%%
\begin{align*}
    \operatorname{concat}&\colon (\R^m)^\npoints \rightarrow \R^{\npoints m}, \quad Y \mapsto \left(y_1^{(1)}, \ldots, y_m^{(1)}, \ldots, y_1^{(\npoints)}, \ldots, y_m^{(\npoints)}\right)^\top, \\
    \operatorname{concat}^{-1}&\colon \R^{\npoints m} \rightarrow (\R^m)^\npoints, \quad  \BF{Y} \mapsto \left\{(Y_1, \ldots, Y_m)^\top, (Y_{m+1}, \ldots, Y_{2m})^\top, \ldots, (Y_{(\npoints-1)m + 1}, \ldots, Y_{\npoints m})^\top\right\}.
\end{align*}
%%%%
The concatenation map gives rise to a hypervolume function that takes vectors in $\R^{\npoints m}$ as input:
%%%%
\begin{equation}
\HVY\colon \R^{\npoints m}\rightarrow \R_{\geq 0}, \quad \Y \mapsto \left[\operatorname{HV} \circ \operatorname{concat}^{-1}\right](\Y), \label{eq:hvy}
\end{equation}
Similarly, we also consider a finite set of decision points (of size $\npoints$) as a single point in $\R^{\npoints d}$, i.e., $\X = [{\BF{x}^{(1)}}^\top, {\BF{x}^{(2)}}^\top, \ldots, {\BF{x}^{(\npoints)}}^\top]^\top \in \R^{\npoints d}$. 
In this sense, the objective function \f is also extended to:
%%%
$$
\F\colon \domain^\npoints \rightarrow \mathbb{R}^{\npoints m}, \X \mapsto [\f(X_1, \ldots, X_{d}), \f(X_{d+1}, \ldots, X_{2d}), \ldots, \f(X_{(\npoints - 1)d+1}, \ldots, X_{\npoints d})]^\top.
$$
%%%
Given \F, we have the relation $\Y =\F(\X)$. Now, we can express hypervolume indicator as a function supported on $\R^{\npoints d}$: 
%%%%
\begin{equation}
    \HVF\colon \mathbb{R}^{\npoints d} \rightarrow \R_{\geq 0},\quad \X \mapsto \left[\HV \circ \operatorname{concat}^{-1} \circ\F\right](\X). \label{eq:hvf}
\end{equation}
%%%%
Notably, assuming $\f$ is twice differentiable, the above hypervolume functions $\HVY$ and \HVF are twice differentiable almost everywhere in their domains, respectively. 

It is straightforward to express the gradient of \HVF w.r.t.~\X using the chain rule as reported in our previous works~\cite{EmmerichD12,WangDBE17}: $\nabla \HVF(\X) = (\partial \HVY/\partial \Y)( \partial \F(\X)/\partial\X)$, in which we also discussed the time complexity of computing the hypervolume gradient.

In this work, we investigate the Hypervolume Indicator Hessian Matrix for more than two objectives and propose an algorithm to compute it efficiently. The work is structured as follows: In Section~\ref{sec:chain}, we briefly review the general construction of the Hypervolume Gradient and Hypervolume Hessian $\HVF$ via the chain-rule as it has been outlined previously in~\cite{EmmerichD12} and, respectively, in~\cite{Sosa-HernandezS20}. Section~\ref{sec:3d} provides a discussion of the 3-D Hypervolume indicator Hessian matrix $\HV$, including a $O(n \log n)$ dimension sweep algorithm for its asymptotically optimal computation and an analysis of its sparsity, i.e., the number of its non-zero components. Furthermore, it is argued that in the 3-D case, the Hypervolume Hessian Matrix is sparse and has at most $O(\npoints)$ non-zero components.
In Section~\ref{sec:general-computation}, we discuss the analytical formulations of the hypervolume Hessian for the general case of $m>1$ objective functions. The result reduces the computation of the Hessian of $m$ objective functions to the repeated computation of the hypervolume indicator gradient for collections of vectors in $\mathbb{R}^{m-1}$. A \texttt{Python} implementation is provided for the general cases.
In Section~\ref{sec:numerical}, we provide numerical examples. In Section~\ref{sec:disuss}, we finish the paper with a discussion of some basic properties of the hypervolume Hessian matrix, such as its continuity, one-sidedness, and rank, and point out interesting open questions for its further analysis.

\section{General Construction of Hypervolume Hessian and Gradient via the Chain Rule}
\label{sec:chain}
Taking the $\R^{\npoints d}$-vector \X and $\R^{\npoints m}$-vector $\Y=\F(\X)$, we express the Hessian matrix of the hypervolume indicator as follows:
%%%%%
\begin{align}
    \nabla^2\HVF &= \frac{\partial}{\partial \X}\left(\frac{\partial \HVY}{\partial \Y} \frac{\partial \Y}{\partial \X}\right)
    =\left[\frac{\partial}{\partial \X}\left(\frac{\partial \HVY}{\partial \Y}\right)\right]^\top\frac{\partial \F(\X)}{\partial \X} + \frac{\partial \HVY}{\partial \Y}\frac{\partial^2 \F(\X)}{\partial \X\partial\X^\top}\nonumber\\
    &=\nabla{\F(\X)}^\top \frac{\partial^2 \HVY}{\partial \Y\partial\Y^\top}\nabla{\F(\X)} + \frac{\partial \HVY}{\partial \Y}\frac{\partial^2 \F(\X)}{\partial \X\partial\X^\top}.\label{eq:HV-Hessian}
\end{align}
%%%%%
The Hessian of vector-valued objective function $\F$, i.e., $\partial^2 \F /\partial \X\partial\X^\top \colon \R^{\npoints d} \rightarrow \Hom(\R^{\npoints d}, \Hom(\R^{\npoints d}, \R^{\npoints m}))$, is a tensor of $(1, 2)$ type.
Let $T^k_{i,j}=\partial^2 F_k /\partial \mathrm{X}_i\partial \mathrm{X}_j, i, j\in[1..\npoints d], k\in[1..\npoints m]$, we specify the entries of $T$ as follows:
%%%%%
\begin{align*}
    &T^{k}_{i,j} = 
\begin{cases}
\partial^2 f_\beta(\BF{x}^{(\alpha)})/\partial x^{(\alpha)}_{i'}\partial x^{(\alpha)}_{j'}, &\text{if } (\alpha-1)d + 1 \leq i, j \leq \alpha d, \\
0,  &\text{otherwise.}
\end{cases}\\
&\alpha=\lceil k / m \rceil, \beta=k - (\alpha - 1)m, i' = i - (\alpha - 1)d, j' = j - (\alpha - 1) d.
\end{align*}
%%%%%
Since the above map from $k$ to $(\alpha, \beta)$ is bijective (its inverse is $k=\alpha\beta$), we will equivalently use $\alpha\beta$ for the contravariant index $k$. It is obvious that tensor $T$ is sparse, where for each $k$, only $d^2$ entries are nonzero, giving up to $\npoints m d^2$ nonzero entries in total.
Using the Einstein summation convention, we can expand the second term in Eq.~\eqref{eq:HV-Hessian} as:
%%%%%
\begin{equation}
   \left(\frac{\partial \HVY}{\partial \Y}\frac{\partial^2 \F(\X)}{\partial \X\partial\X^\top}\right)_{i, j} = \left(\frac{\partial \HVY}{\partial \Y}\right)_k T^{k}_{i, j} = \frac{\partial \HVY}{\partial f_\beta(\BF{x}^{(\alpha)})}T^{\alpha\beta}_{i,j}, \label{eq:term2}
\end{equation}
%%%%%
where $(\cdot)_{i,j}$ denotes the $(i,j)$-entry of a tensor. We have discussed the analytical expression of the term $\partial \HVY/\partial \Y$ in our previous works~\cite{EmmerichD12,WangDBE17}.
Notably, the above expression leads to a \emph{block-diagonal matrix} containing $\npoints$ matrices of shape $d\times d$ on its diagonal. Therefore, we observe a high sparsity of the second term in Eq.~\eqref{eq:HV-Hessian}. As for the first term, $\partial^2 \HVY /\partial \Y\partial\Y^\top$ denotes the Hessian of the hypervolume indicator w.r.t.~objective vectors, whose computation and sparsity will be discussed in the following sections.\\
In our previous work~\cite{Sosa-HernandezS20}, we have derived the analytical expression of $\nabla^2 \HVF$ for bi-objective cases and analyzed the structure and properties of the hypervolume Hessian matrix. Also, we implemented a standalone Hypervolume Newton (HVN) algorithm for unconstrained MOPs. Moreover, we have shown that the Hessian $\nabla^2\HVF$ is a tri-diagonal block matrix in bi-objective cases and provided the non-singularity condition thereof, which states the Hessian is only singular on a null subset of $\R^{\npoints d}$~\cite{Sosa-HernandezS20}, thereby ascertaining the safety of utilizing hypervolume Hessian matrix, e.g., in the hypervolume Newton method~\cite{WangEDHS22} for equality constraints.

\section{Hypervolume Indicator Hessian Matrix in 3-D}
\label{sec:3d}
As with many problems related to Pareto dominance, the 2-D and 3-D cases have a special structure that can be exploited for formulating asymptotically efficient dimension sweep algorithms~\cite{KungLP75,paquete2022computing}. Next, the dimension sweep technique will be applied to yield an asymptotically efficient algorithm for the problem of computing the Hessian Matrix of the 3-D Hypervolume Indicator $\HV$. 
The basic idea is sketched in Fig.~\ref{fig:HV3Di} and consists of computing first the facets of the polyhedron that is measured by the hypervolume indicator by lowering a sweeping plane along each one of the axes. The gradient components are given by the areas of the facets (e.g., the yellow shaded area in Fig.~\ref{fig:HV3Di}, and the \emph{length of the line segments of the ortho-convex polygon that surrounds this facet determines the components of the Hessian matrix of $\HV$}. This can be easily verified by studying geometrically the effect of small perturbations of the coordinates of points in $\mathbf{Y}$ along the coordinate axis on the value of the hypervolume indicator (gradient components) $\HV$.
%%%%%%
\begin{figure}[t]
%\begin{minipage}[b]{.46\textwidth}
\centering
\includegraphics[width=.75\textwidth, trim=0mm 0mm 0mm 2mm, clip]{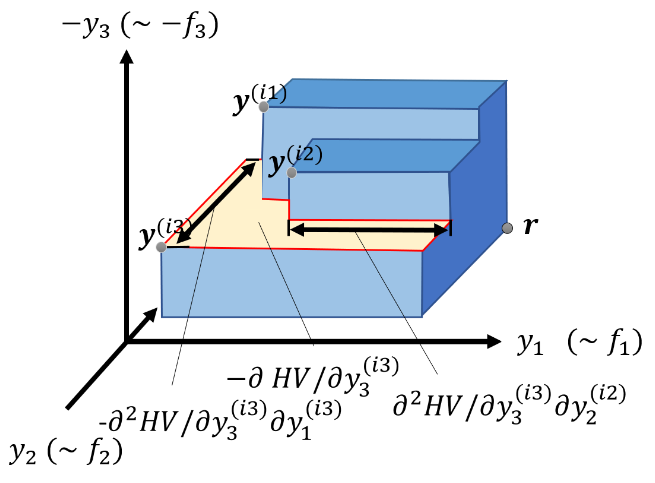}
\captionof{figure}{\label{fig:HV3Di}Visualization of some first and second-order derivatives of a 3-D Hypervolume approximation set $(\BF{y}^{(1)},\BF{y}^{(2)}, \BF{y}^{(3)} )$ with $i1$, $i2$, $i3$ chosen such that $y^{(i1)}_3$ $>$ $y^{(i2)}_3$ $>$ $y^{(i3)}_3$.}
%\end{minipage}
%\hfill
%\begin{minipage}[b]{.52\textwidth}
\end{figure}
%%%%%%
\begin{figure}[t]
\centering
\begin{tikzpicture}[scale=.95]
% Draw the grid
\tikzset{help lines/.style={color=blue!50}}
\draw[thick,step=1cm,help lines] (0,0) grid (9,9);
\draw[ultra thin,step=.5cm,help lines] (0,0) grid (9,9);
% Draw axes
\draw[ultra thick,-latex] (0,0) -- (10,0);
\draw[ultra thick,-latex] (0,0) -- (0,10);
% the co-ordinates -- major
\foreach \x in {0,1,...,9} {     % for x-axis
\draw [thick] (\x,0) -- (\x,-0.2);
}
\foreach \y in {0,1,...,9} {   %% for y-axis
\draw [thick] (0,\y) -- (-0.2,\y);
}
% the numbers
\foreach \x in {0,1,...,8} { \node [anchor=north] at (\x,-0.3) {\x}; }
\foreach \y in {0,1,...,8} { \node [anchor=east] at (-0.3,\y) {\y}; }
% the co-ordinates -- minor
\foreach \x in {.5,1.5,...,8.5} {
\draw [thin] (\x,0) -- (\x,-0.1);
}
\foreach \y in {.5,1.5,...,8.5} {
\draw [thin] (0,\y) -- (-0.1,\y);
}
\filldraw
(-1,9) circle (3pt) node[anchor=south west] {$\mathbf{y}^{(0)}$} 
(2,6) circle (3pt) node[anchor=north east, xshift=8mm] {$\mathbf{y}^{(i5)}=\mathbf{d}_t^{(0)}$}
(3,2) circle (3pt) node[anchor=north east] {$\mathbf{y}^{(i7)}$}
(7,1) circle (3pt) node[align=left,   below] {$\mathbf{y}^{(i4)}=\mathbf{d}_t^{(N_{t}+1)}$}
(4,5) circle (3pt) node[anchor=south west] {$\mathbf{y}^{(i3)}=\mathbf{d}_t^{(1)}$}
(5,3) circle (3pt) node[align=left,   below] {$\mathbf{y}^{(i2)}=\mathbf{d}_t^{(2)}$}
(6,4) circle (3pt) node[anchor=south west] {$\mathbf{y}^{(i1)}$}
(3,6) circle (3pt) node[anchor=south west] {$\mathbf{t}^{(2)}$}
(7,2) circle (3pt) node[anchor=south west] {$\mathbf{t}^{(1)}$}
(9,9) circle (3pt) node[anchor=south west] {$\mathbf{r}$}
(9,-1) circle (3pt) node[anchor=south west] {$\mathbf{y}^{(n+1)}$}
;
\draw [ultra thick,-latex, draw=black,fill=yellow,fill opacity=0.2] (3,2) -- (9,2) -- (9,9) -- (3,9) -- cycle;
\draw [ultra thick,-latex, draw=black,fill=blue,fill opacity=0.2] (9,9) -- (3,9) -- (3,6) -- (4,6) -- (4,5) -- (5,5) -- (5,3) -- (7,3) -- (7,2) -- (9,2) -- cycle;
\draw [ultra thick, dotted, draw=black] (9,-1) -- (9,0);
\draw [ultra thick, dashed, draw=black] (9,0) -- (9,2);
\draw [ultra thick, dotted, draw=black] (-1,9) -- (0,9);
\draw [ultra thick, dashed, draw=black] (0,9) -- (3,9);
\draw [ultra thick, dashed, draw=black] (3,6) -- (2,6) -- (2,9);
\draw [ultra thick, dashed, draw=black] (0,9) -- (3,9);
\draw [ultra thick, dashed, draw=black] (7,2) -- (7,1) -- (9,1);
\end{tikzpicture}
%\vspace{-3mm}
\captionof{figure}{\label{fig:HVDimSweepIteration}Snapshot of the components of the polygon (yellow shaded), in which the length of each edge constitutes the non-zero components of the Hessian matrix of $\HV$ in a single iteration (lowering of sweeping plane) $t$  of Alg.~\ref{alg:Hessian-3D}.}
%\end{minipage}
\end{figure}
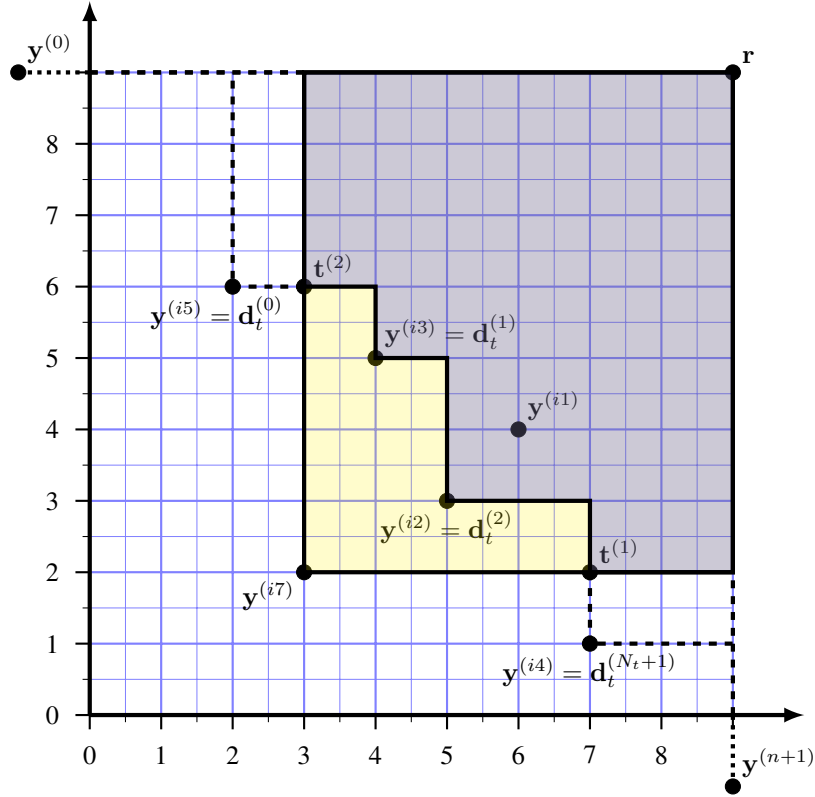
%%%%%%

Without loss of generality, we first compute the derivatives with respect to $y_3$. By permuting the roles of $y_1$, $y_2$, and $y_3$, we can get all derivatives. 
In the context of the dimension sweep algorithm, we assume that points in Fig.~\ref{fig:HV3Di} are sorted by the 3rd coordinate $y_3$, that is, $Y$ is represented in such a way that $y^{(i3)}_3 < y^{(i2)}_3 < y^{(i1)}_3$. We assume that the points in $\mathbf{Y}$ are in general position (otherwise, one-sided derivatives can occur, which will be discussed later).

Alg.~\ref{alg:Hessian-3D} computes all positive entries of the Hessian matrix of $\HV$ at a point $\mathrm{Y}$.
The algorithm proceeds in three sweeps. The sweep coordinate has index $h$ (like height), and the other two coordinates are termed $l$ and $w$ (like length and width). The first sweep sets $h=3$, the second sweep $h=2$, and the third sweep $h=1$. The values of $l$ and $w$ are set to the remaining two coordinates. The roles of $l$ and $w$ are interchangeable, but here we set them to $l=1$, $w=2$ in the first sweep, to $l=1, w=3$ in the second sweep, and to $l=3, w=2$ in the third sweep.
Next, we describe a single sweep in detail. Without loss of generality, let us choose the sweep along the 3-rd coordinate, e.g., $l=1, w=2, h=3$:
First, we introduce sentinels $\BF{y}^{(0)}$ and $\BF{y}^{(n+1)}$ that make sure that every point always has a left and a lower neighbor.
We use a balanced binary search tree $\mathrm{T}$ to efficiently maintain a list of all non-dominated points in the $lw$-plane among the points that have been visited in a single sweep so far.
The sorted list represented by tree $\mathrm{T}$ is initialized by the sentinels; note that the sentinels cannot become dominated in the $lw$-plane because one of their coordinates is $-\infty$. The other coordinates are from the reference point. 
Next, start the loop that starts from the highest $y_h$ coordinate and lowers the sweeping plane to the next highest $y_h$ coordinate in each iteration $t$. The value of $t$ is thus the index of the point that is currently processed, and points are sorted in the $h$-direction using the index transformation $a[t], t = 1,\ldots, n$.
In each iteration, we determine from the sorted list the sublist starting from $\BF{y}^{d[t][0]}$ and terminating with $\BF{y}^{d[t][N_t]}$. We assume the list is sorted ascendingly in the $l$-coordinate. The point is $\mathrm{T}$with the highest $y_l$ coordinate that does not exceed $y^{a[t]}_l$ is chosen as $\bf{t}^{d[t][0]}$ and the point with the highest $y_w$ coordinate that does not yet exceed $y^{(a[t])}_w$ is chosen as $\bf{t}^{(d[t][N_t+1])}$. These two points always exist because of the sentinels we set initially. The points between these points in the list, given there exist such points, are referred to by $\BF{y}^{(d[t][1])}$, $\dots$, $\BF{y}^{(d[t][N_t])}$. If no such points exist, then $N_t$ is set to $0$.
Note, that the points $\BF{y}^{(d[t][1])}, \dots, \BF{y}^{(d[t][N_t])}$ are points that become dominated by $\BF{y}^{(a[t])}$ in the $lw$-projection. They will be discarded from $\mathrm{T}$ at the end of the iteration, and $\BF{y}^{(a[t])}$ will be inserted to $\mathrm{T}$ thereafter so that the list represented by $\mathrm{T}$ remains a list of mutually non-dominated points in the $lw$-projection.
Before discarding the points from $\mathrm{T}$, the new positive components of the Hypervolume Hessian are computed. This is done by computing the line segments of the polygonal area that is marked by the points $\BF{y}^{(a[t])}$ and $\BF{y}^{(d[t][0])}$, $\dots$, $\BF{y}^{(d[t][N_t+1])}$ as it is indicated graphically in Fig.~\ref{fig:HVDimSweepIteration} for a single iteration of the algorithm. This is the polygonal region that marks the area of the hypervolume gradient $\partial\HVY / y^{(a[t)]}_h$. Changing the coordinates of the corners of this polygon infinitesimally adds a differential change to the area, which is the aforementioned hypervolume indicator gradient component. The details of the assignment to the Hypervolume Hessian can be determined by computing the side-lengths of the edges of the polygon and carefully tracing which coordinates of points in $\BF{Y}$ determine the coordinates of the region the area of which determines the gradient component (the yellow area in Fig.~\ref{fig:HV3Di} and in Fig.~\ref{fig:HVDimSweepIteration}).   

%%%%%%
\begin{theorem}{Computation of Hessian Matrix components by \textsc{HV3D-TriSweep} algorithm}
\label{thm:3Dtrisweepsoundness}
Assume that $n$ mutually non-dominated points in $\mathbb{R}^3$ are given by a collection $\BF{Y}$, and assume they are in general position (no duplicate coordinates). Furthermore, assume that all points in $\BF{Y}$ dominate the reference point $\mathbf{r}$.
Then algorithm \ref{alg:Hessian-3D} that we will term \textsc{HV3D-TriSweep} computes all non-zero components of the Hessian matrix $\HV$.
\end{theorem}

\begin{proof}
The idea of the Alg.~\ref{alg:Hessian-3D} is that it computes the polygons that mark the region, the area of which is the hypervolume derivative. It uses the same dimension sweep approach then detailed in~\cite{EmmerichD12} to compute these polygons, one polygon at each lowering of the sweep-plane.
Now, it can be easily derived by geometrical considerations on the changes in the coordinates of corner points differentially change the area of the polygon. The corner points of the polygon are coordinates of points in $\BF{Y}$, and thus the second partial derivatives will be related to these coordinates. Due to the axis aligned, orthogonal geometry of the polygon, these second derivatives are easy to establish as the length of the edges of the polygon.
\end{proof}

Next, we analyze the time complexity of Alg.~\ref{alg:Hessian-3D} and study the number of non-zero components it computes, which corresponds exactly to the non-zero components of the Hessian matrix of $\HV$.
To do so, let us first proof a lemma:
%%%%%
\begin{algorithm2e}[t]
\caption{HVH3D Triple Dimension Sweep algorithm to Compute all non-zero components of $\partial^2\HVY/\partial{\mathbf{Y}}\partial{\BF{Y}^\top}$}\label{alg:Hessian-3D}
\textbf{Procedure:} HVH3D-TriSweep($Y$)\;
\textbf{Input:} a collection of vectors $Y=\{\mathbf{y}^{(1)}, \dots, \mathbf{y}^{(n)}\}\subset \mathbb{R}^m$, reference point $\mathbf{r} \in \mathbb{R}^m$\;

\textbf{Output:} Non-zero components of the Hessian matrix of $\HV$: $\partial^2\HVY/\partial y^{(i)}_{m} \partial y^{(j)}_{k}$\;

\tcp{Three sweeps along the different coordinate axis.}
 
\For{$(l,w,h) \in ((1,2,3), (1,3,2), (3,2,1))$}{
    $y^{(0)}_l = -\infty, y^{(0)}_w = r_w, y^{(0)}_h = r_h$ \Comment*[r]{Define the sentinels}
    
    $y^{(n+1)}_l=r_l, y^{(n+1)}_l= -\infty, y^{(n+1)}_l = r_h$\;
    
    $(\mathbf{y}^{(a[1])},\dots,\mathbf{y}^{(a[n])}) \leftarrow$ Sort $(\mathbf{y}^{(1)}, \dots, \mathbf{y}^{(n)})$ descendingly by $y_h$ coordinate\;
    
    $\mathrm{T}=\mathrm{BSTree}((\mathbf{y}^{(0)}, \mathbf{y}^{(n+1)}))$ \Comment*[r]{Initialize balanced search tree}    
    \For{$t \in (1,\dots, n)$}     
    { 
        \Comment{$t$ is the position of  the sweeping plane along $y_h$ axis}
    
        Determine $N[t]$ and $(\mathbf{y}^{(d[t][0])}, \dots, \mathbf{y}^{(d[t][N[t]+1])})$ based on $\mathbf{y}^{(a[t])}$ as sublist of $\mathrm{T}$ starting from nearest lower neighbor to $\mathbf{y}^{(a[t])}$ in $l$ direction and terminating at nearest lower neighbor of $\mathbf{y}^{(a[t])}$ in $w$-direction\;

        $\partial^2\HVY/\partial y^{(a[t])}_{h} \partial y^{(a[t])}_{l} = - (y^{(d[t][0])}_w- y^{(a[t])}_w)$\;
      
        $\partial^2\HVY/\partial y^{(a[t])}_{h} \partial y^{(a[t])}_{w} = - (y^{(d[t][N_t])}_l- y^{(a[t])}_l)$\;
      
        \If{$N[t]>0$}{
            $\partial^2\HVY/\partial y^{(a[t])}_{h} \partial y^{(d[t][0])}_{w} = y^{(d[t][1])}_l - y^{(a[t])}_l$\;
            
            $\partial^2\HVY/\partial y^{(a[t])}_{h} \partial y^{(d[t][N_t])}_l = y^{(d[t][N_t-1])}_w - y^{(a[t])}_w$\;
            
            \For{$j=1,\ldots, N_t$}{
                $\partial^2\HVY/\partial y^{(a[t])}_{h} \partial y^{(d[t][N_t])}_l = y^{(d[t][j-1])}_w - y^{(d[t][j])}_w$\;
                
                $\partial^2\HVY/\partial y^{(a[t])}_{h} \partial y^{(d[t][N_t])}_w = y^{(d[t][j+1])}_l - y^{(d[t][j])}_l$\;
            }
            Discard $(\mathbf{y}^{(d[t][1])}, \dots,   \mathbf{y}^{(d[t][N[t]])})$ from $\mathrm{T}$\;
            
            Add $\mathbf{y}^{a[t]}$ to tree $\mathrm{T}$\;
        }
    }
}
\textbf{return} $\partial^2\HVY/\partial y^{(i)}_{m} \partial y^{(j)}_{k}$ \Comment*[r]{return only the $O(n)$ computed elements and their indices}
\end{algorithm2e}
%%%%%
\begin{lemma}
In algorithm \ref{alg:Hessian-3D} it holds that $\sum_{t}^n N_t = n-1$.
\label{lem:amort}
\end{lemma}

\begin{proof}
$N_t$ is the number of dominated points in the 2-D plane spanned by the $y_l$ and $y_w$ coordinate in one iteration of the algorithm marked by iteration index $t$.
Each point in $\mathbf{Y}$ gets dominated in the 2-D projection at most once, and in this case, it contributes to $N_t$ and is removed from the tree, with the exception of the last point.
\end{proof}

\begin{theorem}{Time Complexity of 3-D Hessian Matrix of $\HV$}
\label{thm:3Dtime}
The computation of all non-zero components of the Hessian matrix of the mapping from a set $\mathbf{Y} \in \mathbb{R}^3$ in the objective space to the hypervolume indicator takes computational time $\Theta(n \log n)$. 
\end{theorem}

\begin{theorem}{Sparsity and Space of 3-D Hessian Matrix of $\HV$}
\label{thm:3Dsparsity}
The number of all non-zero components of the Hessian matrix of the mapping from a set $\mathbf{Y} \in \mathbb{R}^3$ in the objective space to the hypervolume indicator does never exceed $12n-6$.
\end{theorem}
%%%%%
\begin{proof}
    The algorithm produces in each step $t$ exactly $2N_t+4$ entries of the Hypervolume Hessian Matrix. Per dominated point, there are $2$ non-zero components computed. There are exactly $4$ additional non-zero entries computed in each iteration $t$. Therefore the total computation time amounts to $2 N_t$ steps to process the newly dominated points in the 2-D projection, plus the $4$ additional components of the polytope boundary in each iteration. By amortization described in Lemma~\ref{lem:amort} of $N_t$, we obtain $2(n-1)+4n$ components per sweep, and we carry out $3$ sweeps, resulting in at most $12n-6$ non-zero components. 
\end{proof}

\section{General Expression of the N-Dimensional Hypervolume Hessian Matrix} \label{sec:general-computation}
For the general cases ($m>3$), it suffices to compute the term $\partial^2 \HVY/\partial \Y \partial \Y^\top$ and utilize Eq.~\eqref{eq:HV-Hessian} for computing the hypervolume Hessian. We summarize the computation in Alg.~\ref{alg:general} and explain the details as follows. Let $\BF{A} =\partial^2 \HVY/\partial\Y \partial\Y^\top \in \R^{\npoints m \times \npoints m}$. Without loss of generality, we calculate the entries of $\BF{A}$ in a column-wise manner - given indices $i\in[1..\npoints]$ and $k\in[1..m]$, column $ik$ of $\BF{A}$ takes the following form:
$$
\left(\frac{\partial}{\partial\BF{y}^{(1)}}\left(\frac{\partial\HVY}{\partial y_k^{(i)}}\right)^\top, \ldots,
\frac{\partial}{\partial\BF{y}^{(i)}}\left(\frac{\partial\HVY}{\partial y_k^{(i)}}\right)^\top, \ldots, \frac{\partial}{\partial\BF{y}^{(\npoints)}}\left(\frac{\partial\HVY}{\partial y_k^{(i)}}\right)^\top\right)^\top,
$$
which will be discussed in two scenarios: (1) $\partial(\partial \HVY/\partial y_k^{(i)})/\partial \BF{y}^{(i)}$ and (2) $\partial(\partial \HVY/\partial y_k^{(i)})/\partial \BF{y}^{(j)}, i\neq j$.  
The dominated space of $S\subseteq\R^m$, that is the subset of $\R^m$ dominated by the points in $S$, is denoted by $\Dom{S} = \{\BF{p}\in\R^m\colon \exists \BF{y}\in S(\BF{y}\prec \BF{p}) \wedge \BF{p}\prec \r\}$. Also, we take the canonical basis $\{\BF{e}_i\}_i$ of Euclidean spaces in our elaboration. 

%%%%
\begin{algorithm2e}[t]
\caption{General algorithm for the hypervolume Hessian matrix}\label{alg:general}
\textbf{Input:} $X=\{\BF{x}^{(1)}, \ldots, \BF{x}^{(\npoints)}\}$: decision points, $Y=\{\BF{y}^{(1)}, \ldots, \BF{y}^{(\npoints)}\}$: objective points, $\nabla\F, \nabla^2\F$: Jacobian and Hessian of the objective function\;
 \textbf{Output:} $\BF{H}$: hypervolume Hessian matrix\;
 $\BF{X}\leftarrow \operatorname{concat}(X)$\;
$\BF{A} \leftarrow \BF{0}_{\npoints m \times \npoints m}$\Comment*[r]{$\BF{A} \coloneqq \partial^2 \HVY/\partial \Y \partial \Y^\top$}
\For{$i = 1,\ldots, \npoints$}{
    \lIf{$\BF{x}^{(i)}$ is dominated}{\textbf{continue}}
    \For{$k = 1,\ldots, m$}{
        $\BF{y}_{\shortminus k}^{(i)}\leftarrow \proj_k(\BF{y}^{(i)})$;
        
        $Y' \leftarrow \{\proj_k(\BF{y}) \colon \BF{y}\in Y \wedge y_k < y^{(i)}_k\}$\;
        $I(Y,i,k)\leftarrow \{\alpha\in [1..n]\colon y^{(\alpha)}_{k} < y^{(i)}_{k}\}$\;
        
        \tcp{compute $\partial(\partial \HVY/\partial y_k^{(i)})/\partial \BF{y}^{(i)}$}
        \For{$l=1,\ldots,m$}{
        	\lIf{$l=k$}{$v_l\leftarrow 0$; \textbf{continue}}
        	$p \leftarrow l \text{ if } l < k; \text{otherwise } p \leftarrow l-1$\;
        	$v_l \leftarrow \HVC\left(\proj_{p}(\BF{y}^{(i)}_{\shortminus k}), \left\{\proj_{p}(\BF{y}^{(\alpha)}_{\shortminus k})\colon \alpha\in I(Y', i, l)\right\}\right)$\Comment*[r]{$v_l \coloneqq \partial^2 \HVY/\partial y^{(i)}_l \partial y_k^{(i)}$; Eq.~\eqref{eq:partial-ii-cross-term}}
        }

    	\For{$s = im + 1,\ldots, (i+1)m$}{
    		$l\leftarrow s - im$\;
    	    $\BF{A}_{s, im+k} \leftarrow v_l$\;
    	}
    	
    	\tcp{compute $\partial(\partial \HVY/\partial y_k^{(i)})/\partial \BF{y}^{(j)}$}
    	$\widehat{Y} \leftarrow \left\{\textsc{clip}(\BF{y};\BF{y}_{\shortminus k}^{(i)})\colon \BF{y} \in Y' \right\}$\Comment*[r]{clipping operation; Eq.~\eqref{eq:clipping}}
    	
    	\For{$j \in I(Y,i,k)$}{
    	    $\widehat{\BF{y}}_{\shortminus k}^{(j)} \leftarrow \widehat{Y}[j]$\Comment*[r]{take element $j$ from set $\Y'$}

    		\For{$l=1,\ldots,m$}{
        		\lIf{$l=k$}{$w_l\leftarrow 0$; \textbf{continue}}
        		$p \leftarrow l \text{ if } l < k; \text{otherwise } p \leftarrow l-1$\;
        		$w_l \leftarrow -\HVC\left(\proj_{p}(\widehat{\BF{y}}^{(j)}_{\shortminus k}), \left\{\proj_{p}(\widehat{\BF{y}}^{(\alpha)}_{\shortminus k})\colon \alpha\in I(
        		\widehat{Y}, j, p)\right\}\right)$\Comment*[r]{$w_l \coloneqq \partial^2 \HVY/\partial y^{(j)}_l \partial y_k^{(i)}$}
        	}
	
    		\For{$s = jm + 1,\ldots, (j+1)m$}{
    			$l\leftarrow s - jm$\;
    		    $\BF{A}_{s, im+k} \leftarrow w_l$\;
    		}
    	}
    }
}
$T \leftarrow \nabla^2\F(\X)$\;
$\BF{H} \leftarrow \nabla{\F}(\X)^\top\BF{A}\nabla{\F}(\X) + \sum_{\alpha=1}^{\npoints}\sum_{\beta=1}^{m} \left(\partial \HVY/\partial f_\beta(\BF{x}^{(\alpha)})\right) T^{\alpha\beta}$\Comment*[r]{Eq.~\eqref{eq:HV-Hessian}}
\end{algorithm2e}
%%%%%
\subsection{Partial derivative $\partial(\partial \HVY/\partial y_k^{(i)})/\partial \BF{y}^{(i)}$} 
Intuitively, from Fig.~\ref{fig:HV3Di}, we observe that $\partial^2 \HVY/\partial y_k^{(i)}\partial y_k^{(i)}$ is always zero for the 3D case since $\partial \HVY/\partial y_k^{(i)}$ is essentially the hypervolume improvement of the projection of $\BF{y}^{(i)}$ along axis $\BF{e}_k$ (bright yellow area in Fig.~\ref{fig:HV3Di}), ignoring the points that dominates $\BF{y}^{(i)}$ after the projection. In addition, $\partial^2 \HVY/\partial y_k^{(i)}\partial y_\alpha^{(i)}, \alpha\neq i$ equals the negation of partial derivative of the hypervolume indicator w.r.t.~$y_\alpha^{(i)}$ in the $m-1$-dimensional space, resulted from the projection. The computation of this quantity has been investigated in great detail previously~\cite{EmmerichD12}. We prove this argument for $m>3$ as follows. First, we define the orthogonal projection operator onto $\BF{e}_k^\perp$:
%%%%%
$$
\proj_k\colon \BF{y} \mapsto (\ldots, y_{k-1}, y_{k+1}, \ldots)^\top,
$$
%%%%%
which drops the $k$-th components of the input point $\BF{y}$. Here, we do not specify the dimensionality of $\BF{y}$ on purpose, since later in the discussion, $\proj_k$ will be applied to points in $\R^m$ or $\R^{m-1}$.
%When applied to a set $S$, $\proj_k(S)$ operates on each element of $S$.
%%%%%
\begin{theorem}[Partial derivative of $\HVY$] \label{thm:partial-derivative-hv} Assume a finite approximation set $Y=\{\BF{y}^{(1)}, \ldots \BF{y}^{(\npoints)}\}\in (\R^m)^{\npoints}$. We define $\BF{y}^{(\alpha)}_{\shortminus k} = \proj_k(\BF{y}^{(\alpha)}), \alpha\in[1..\npoints]$ and the index set $I(Y,i,k) = \{\alpha\in [1..n]\colon y^{(\alpha)}_{k} < y^{(i)}_{k}\}$ that selects the points in $Y$ whose $k$-th component is strictly better/smaller than $y^{(i)}_{k}$. The partial derivative of $\HVY$ w.r.t. $y_k^{(i)}$ for all $i\in[1..\npoints]$ and $k\in[1..m]$ admits the following expression:
$$
\frac{\partial \HVY(\Y)}{\partial y_k^{(i)}} = -\HVC\left(\BF{y}^{(i)}_{\shortminus k}, \left\{\BF{y}^{(\alpha)}_{\shortminus k}\colon \alpha\in I(Y,i,k)\right\}\right),
$$
where $\Y = \operatorname{concat}(Y)$ and $\HVC(\BF{y}, P)$ is the hypervolume contribution of point $\BF{y}$ to a finite point set $P$, i.e.,
$$
\HVC(\BF{y}, P) = \HV(P\cup \{\BF{y}\}) - \HV(P) =  \lambda_{m-1}\left(\Dom[m-1]{\BF{y}}\setminus \cup_{\BF{p}\in P}\Dom[m-1]{\BF{p}}\right).
$$
\end{theorem}
%%%%%
\begin{proof}
Without loss of generality, we can assume the points in $Y$ are indexed in the ascending order w.r.t. the $k$-component, i.e., 
$y_k^{(1)} < \cdots y_k^{(i-1)} < y_k^{(i)} < y_k^{(i+1)} \cdots < y_k^{(\npoints)}$. Then, we have $I(Y,i,k) = [1..i-1]$. Assume a small perturbation $\delta$ that satisfies $0 < |\delta| < \min\{y_k^{(i)} - y_k^{(i-1)}, y_k^{(i+1)} - y_k^{(i)}\}$. After adding $\delta$ to $y_k^{(i)}$, we shall denote the resulting range of $y_k^{(i)}$ by $R(\delta) = [y_k^{(i)}, y_k^{(i)} + \delta]$ if $\delta >0$, or $R(\delta) = [y_k^{(i)} - \delta, y_k^{(i)}]$ if $\delta < 0$. Then, the dominated space of $\BF{y}^{(i)}$ will expand ($\delta < 0$) or shrink ($\delta > 0$) by the follow set:
$$
S(\delta) = R(\delta) \times \Dom[m-1]{\BF{y}^{(i)}_{\shortminus k}}, 
$$
where $\times$ denotes the Cartesian product. 
Due to the assumption on $\delta$, we observe that (1) $S(\delta)$ is disjoint from $\Dom{\BF{y}^{(\alpha)}}$ for $\alpha\in[i+1...\npoints]$ and (2) for $\alpha\in[1...i-1]$, $S(\delta)\cap\Dom{\BF{y}^{(\alpha)}} \neq \emptyset$; However, $\Dom{\BF{y}^{(\alpha)}}$ will not be affected by this small perturbation on $y_k^{(i)}$. Hence, the change of \HV can only be attributed to $S(\delta)$, which is:
%%%%%%
\begin{align}
    \Delta \HV(\delta) &= -\sign(\delta)\lambda_m\left(S(\delta)\setminus\left(S(\delta) \cap \left(\cup_{\alpha=1}^{\npoints}\Dom{\BF{y}^{(\alpha)}}\right)\right)\right) \nonumber\\
    &=-\sign(\delta)\lambda_m\left(S(\delta)\setminus\left(\cup_{\alpha=1}^{\npoints}\Dom{\BF{y}^{(\alpha)}}\right)\right) \nonumber\\
    &=-\sign(\delta)\lambda_m\left(S(\delta)\setminus\left(\cup_{\alpha\in I}\Dom{\BF{y}^{(\alpha)}}\right)\right) \label{eq:delta-HV1},
\end{align}
%%%%%%
where $\sign$ is the sign function. Noticing that $\forall\BF{p}\in\cup_{\alpha\in I}\Dom{\BF{y}^{(\alpha)}}(p_k\notin R(\delta) \iff \BF{p}\notin S(\delta))$, meaning the only subset of $\R^m$ that intersects $S(\delta)$ is $U = \{\BF{p} \in \cup_{\alpha\in I}\Dom{\BF{y}^{(\alpha)}}\colon p_k \in R(\delta)\}$. Due to the construction of $U$, it admits the following expression: 
$$
U = R(\delta) \times \bigcup_{\alpha\in I}\Dom[m-1]{\proj_k(\BF{y}^{(\alpha)})}.
$$
%%%%
Thereby, we can continue simplifying Eq.~\eqref{eq:delta-HV1} using $U$:
\begin{align}
    \Delta \HV(\delta) &=-\sign(\delta)\lambda_m\left(S(\delta)\setminus U\right) \nonumber\\
    &=-\sign(\delta)\lambda_m\left(R(\delta)\times \left[\Dom{\BF{y}^{(i)}_{\shortminus k}}\setminus \left(\cup_{\alpha\in I}\Dom{\BF{y}^{(\alpha)}_{\shortminus k}}\right)\right]\right) \nonumber\\
    &=-\sign(\delta)\lambda_1\left(R(\delta)\right)\lambda_{m-1}\left(\left[\Dom{\BF{y}^{(i)}_{\shortminus k}}\setminus \left(\cup_{\alpha\in I}\Dom{\BF{y}^{(\alpha)}_{\shortminus k}}\right)\right]\right) \nonumber\\
    &=-\delta\lambda_{m-1}\left(\Dom[m-1]{\BF{y}^{(i)}_{\shortminus k}}\setminus \left(\cup_{\alpha\in [1..i-1]}\Dom[m-1]{\BF{y}^{(\alpha)}_{\shortminus k}}\right)\right) \nonumber \\
    &=-\delta\HVC\left(\BF{y}^{(i)}_{\shortminus k}, \{\BF{y}^{(\alpha)}_{\shortminus k}\colon \alpha\in [1..i-1]\}\right).\nonumber
\end{align}
%%%%
Finally, we have:
$$
\frac{\partial\HVY}{\partial y_k^{(i)}} = \lim_{\delta \to 0}\frac{\Delta \HV(\delta)}{\delta} = -\HVC\left(\BF{y}^{(i)}_{\shortminus k}, \left\{\BF{y}^{(\alpha)}_{\shortminus k}\colon \alpha\in [1..i-1]\right\}\right) = -\HVC\left(\BF{y}^{(i)}_{\shortminus k}, \left\{\BF{y}^{(\alpha)}_{\shortminus k}\colon \alpha\in I(Y,i,k)\right\}\right).
$$
\end{proof}
%%%%%%
\begin{corollary}\label{corollary:diagonal-partial-derivatives}
It follows immediately from Thm.~\ref{thm:partial-derivative-hv} that 
%%%
\begin{equation}\label{eq:partial-ii-diagonal-term}
    \frac{\partial^2 \HVY}{\partial y_k^{(i)}\partial y_k^{(i)}}=0, \quad i\in[1..\npoints], k\in[1..m], m\in\mathbb{N}_{>0}. 
\end{equation}
%%%
For computing $\partial^2 \HVY/\partial y_l^{(i)}\partial y_k^{(i)}$ ($l\neq k$), we recursively apply Thm.~\ref{thm:partial-derivative-hv} to the point set 
$$Y'\coloneqq\{\BF{y}^{(i)}_{\shortminus k}\}\cup\{\BF{y}^{(\alpha)}_{\shortminus k}\colon \alpha\in I(Y,i,k)\}.$$
Also, noticing that $y_l^{(i)}$ is the $l$-th component of $\BF{y}^{(i)}_{\shortminus k}$ if $l<k$, and is the $l-1$-th component otherwise, we define $p = l \text{ if } l < k; \text{otherwise } p = l-1$. Taking the index set $I(Y', i, l) = \{\alpha\in [1..|Y'|]\colon y^{(\alpha)}_{l} < y^{(i)}_{l}\}$, we have:
\begin{equation}
      \frac{\partial^2 \HVY}{\partial y_l^{(i)}\partial y_k^{(i)}}=\HVC\left(\proj_{p}(\BF{y}^{(i)}_{\shortminus k}), \left\{\proj_{p}(\BF{y}^{(\alpha)}_{\shortminus k})\colon \alpha\in I(Y', i, l)\right\}\right), \label{eq:partial-ii-cross-term}
\end{equation}
which is $m-2$-dimensional Lebesgue measure. 
This recursive computation is employed by line 14 of Alg.~\ref{alg:general}.
\end{corollary}
%%%%%
\begin{proof}
Eq.~\eqref{eq:partial-ii-diagonal-term} is obvious since $\partial\HVY/\partial y_k^{(i)}$ is not a function of $y_k^{(i)}$ (due to the $\proj_k$ operation in Thm.~\ref{thm:partial-derivative-hv}).
For Eq.~\eqref{eq:partial-ii-cross-term}, we start from the result of Thm.~\ref{thm:partial-derivative-hv}, which is $\partial\HVY/\partial y_k^{(i)}= -\HVC(\BF{y}^{(i)}_{\shortminus k}, \{\BF{y}^{(\alpha)}_{\shortminus k}\colon \alpha\in I(Y,i,k)\})$. Note that when the perturbation of $y_{l}^{(i)}$ ($l\neq k$) is sufficiently small, the resulting change of the hypervolume contribution of $\BF{y}^{(i)}_{\shortminus k}$ to $\{\BF{y}^{(\alpha)}_{\shortminus k}\colon \alpha\in I(Y,i,k)\}$ is exactly the same as the change of \HV of the set $Y'\coloneqq\{\BF{y}^{(i)}_{\shortminus k}\}\cup\{\BF{y}^{(\alpha)}_{\shortminus k}\colon \alpha\in I(Y,i,k)\}$. Noticing that $y_l^{(i)}$ is the $l$-th component of the projected point $\BF{y}^{(i)}_{\shortminus k}$ if $l<k$ and is the $l-1$-th component otherwise, we define $p = l \text{ if } l < k; \text{otherwise } p = l-1$. Hence, when applying Thm.~\ref{thm:partial-derivative-hv} again on $Y'$, we have to project all points in $Y'$ onto $\BF{e}_{p}^\perp$. Now, we can simplify $\partial(\partial \HVY/\partial y_k^{(i)})/\partial y_l^{(i)}, l\neq i$ as follows:
%%%
\begin{align*}
    \frac{\partial^2 \HVY}{\partial y_l^{(i)}\partial y_k^{(i)}} &= -\frac{\partial}{\partial y_{p}^{(i)}}\left(\HVC\left(\BF{y}^{(i)}_{\shortminus k}, \{\BF{y}^{(\alpha)}_{\shortminus k}\colon \alpha\in I(Y,i,k)\}\right)\right)\\
    &=-\frac{\partial}{\partial y_{p}^{(i)}}\HVY\left(\operatorname{concat}\left(Y'\right)\right) \\
    &=\HVC\left(\proj_{p}(\BF{y}^{(i)}_{\shortminus k}), \left\{\proj_{p}(\BF{y}^{(\alpha)}_{\shortminus k})\colon \alpha\in I(Y', i, l)\right\}\right).
\end{align*}
%%%%%%
The last step above is obtained by applying Thm.~\ref{thm:partial-derivative-hv} again on the point set $Y'$, which involves projecting the objective points in $\R^{m-1}$ onto $\BF{e}_{p}^\perp$.
\end{proof}
%%%%%%%%
In all, we have elaborated the method to compute $\partial(\partial \HVY/\partial y_k^{(i)})/\partial \BF{y}^{(i)}$, which constitutes $d$ entries of column $ik$ in matrix $\BF{A}$. We proceed to the computation of the remaining entries in the next sub-section.

\subsection{Partial derivative $\partial\left(\partial \HVY/\partial y_k^{(i)}\right)/\partial \BF{y}^{(j)}, i\neq j$} 
Based on Thm.~\ref{thm:partial-derivative-hv}, we conclude that 
%%%%
$
\partial(\partial \HVY/\partial y_k^{(i)})/\partial \BF{y}^{(j)} = 0 \iff j\notin I(Y, i, k), 
$
%%%%
since $\partial \HVY/\partial y_k^{(i)}$ only depends on $\BF{y}^{(i)}_{\shortminus k}$ and $\{\BF{y}^{(\alpha)}_{\shortminus k}\colon \alpha\in I(Y, i, k)\}$. Also, for all $j\in I(Y, i, k)$, we have
%%%%
\begin{equation} \label{eq:partial-ij-diagonal-term}
    \partial^2 \HVY/\partial y_k^{(j)}\partial y_k^{(i)} = 0,
\end{equation}
%%%%
due to the projection operation. Note that, $\partial(\partial \HVY/\partial y_k^{(i)})/\partial \BF{y}^{(j)}$ constitutes the remaining $(\npoints-1)$ entries of column $ik$ of matrix $\BF{A}$, containing at most $(i-1)(d-1)$ nonzero values, where $i$ depends on the number of points which have a smaller $k$th-component than that of $\BF{y}^{(i)}$. Hence, we can bound the nonzero elements in $\BF{A}$ by $ O(\npoints(\npoints-1)/2 (d-1))$ $=O(n^2 d)$. Note that possible a sharper bound can be formulated by considering the technique used to prove Thm.~\ref{thm:3Dsparsity} in more than three dimensions. 

The remaining partial derivatives can be computed by first clipping points in $\{\BF{y}^{(\alpha)}_{\shortminus k}\colon \alpha\in I(Y, i, k)\}$ by $\BF{y}^{(i)}_{\shortminus k}$ from below (line 13 in Alg.~\ref{alg:general}; sub-procedure \textsc{clip}):
$\widehat{\BF{y}}^{(\alpha)}_{\shortminus k} = \textsc{clip}(\BF{y}^{(\alpha)}_{\shortminus k};\BF{y}^{(i)}_{\shortminus k})$, where for $\BF{a},\BF{b}\in\R^{m-1}$, 
%%%%
\begin{equation}\label{eq:clipping}
    \operatorname{CLIP}(\BF{a};\BF{b})=(a_1 + \min\{0, a_1 - b_1\}, \ldots,  a_{m-1} + \min\{0, a_{m-1} - b_{m-1}\}\})^\top.
\end{equation}
%%%%
Note that, (1) the clipping operation restricts points $\BF{y}^{(\alpha)}_{\shortminus k},\alpha\in I(Y,i,k)$ in $\Dom[m-1]{\BF{y}^{(i)}_{\shortminus k}}$, which is crucial to the following steps since otherwise, the infinitesimal perturbation of $\BF{y}^{(\alpha)}_{\shortminus k}$ will not change the hypervolume contribution of $\BF{y}^{(\alpha)}_{\shortminus k}$; (2) this clipping operation does not change the Lebesgue measure of $\Dom[m-1]{\BF{y}^{(i)}_{\shortminus k}}\setminus \left(\cup_{\alpha\in I}\Dom[m-1]{\BF{y}^{(\alpha)}_{\shortminus k}}\right)$.
%since the subsets that clipped out are not in $\Dom[m-1]{\BF{y}^{(i)}_{\shortminus k}}$. 
After taking the clipping operation, we have the following result:
%%%%%%
\begin{align}
    &\HVC\left(\BF{y}^{(i)}_{\shortminus k}, \{\BF{y}^{(\alpha)}_{\shortminus k}\colon \alpha\in I(Y,i,k)\}\right) 
    \nonumber\\
    &= \lambda_{m-1}\left(\Dom[m-1]{\BF{y}^{(i)}_{\shortminus k}}\setminus \left(\cup_{\alpha\in I}\Dom[m-1]{\widehat{\BF{y}}^{(\alpha)}_{\shortminus k}}\right)\right) \nonumber\\
    &=\lambda_{m-1}\left(\Dom[m-1]{\BF{y}^{(i)}_{\shortminus k}}\setminus\Dom[m-1]{\{\widehat{\BF{y}}^{(\alpha)}_{\shortminus k}\colon \alpha\in I(Y,i,k)\}}\right) \nonumber\\
    &=\lambda_{m-1}\left(\Dom[m-1]{\BF{y}^{(i)}_{\shortminus k}}\right) - \lambda_{m-1}\left(\Dom[m-1]{\{\widehat{\BF{y}}^{(\alpha)}_{\shortminus k}\colon \alpha\in I(Y,i,k)\}}\right) \label{eq:intermediate}. 
\end{align}
The last step in the above equation is due to the fact that after clipping, $\Dom[m-1]{\{\widehat{\BF{y}}^{(\alpha)}_{\shortminus k}\colon \alpha\in I(Y,i,k)\}}\subset \Dom[m-1]{\BF{y}^{(i)}_{\shortminus k}}$.
As with corollary~\ref{corollary:diagonal-partial-derivatives}, we define $p = l, \text{ if } l < k; p = l-1, \text{ if } l > k$.
For $j\neq i$, $l\neq k$, we have: 
%%%%%%
\begin{align}
    &\frac{\partial^2 \HVY}{\partial y_l^{(j)}\partial y_k^{(i)}} \nonumber\\ 
    &=-\frac{\partial}{\partial y_{p}^{(j)}}\left(\HVC\left(\BF{y}^{(i)}_{\shortminus k}, \{\BF{y}^{(\alpha)}_{\shortminus k}\colon \alpha\in I(Y, i, k)\}\right)\right)  \nonumber\\
    &=-\frac{\partial}{\partial y_{p}^{(j)}}\left[\lambda_{m-1}\left(\Dom[m-1]{\BF{y}^{(i)}_{\shortminus k}}\right) - \lambda_{m-1}\left(\Dom[m-1]{\{\widehat{\BF{y}}^{(\alpha)}_{\shortminus k}\colon \alpha\in I(Y,i,k)\}}\right)\right] && \text{apply Eq.~\eqref{eq:intermediate}} \nonumber\\
    &=\frac{\partial}{\partial y_{p}^{(j)}}\lambda_{m-1}\left(\Dom[m-1]{\{\widehat{\BF{y}}^{(\alpha)}_{\shortminus k}\colon \alpha\in I(Y, i, k)\}}\right)  \nonumber\\
    &=\frac{\partial}{\partial y_{p}^{(j)}}\HVY\left(\operatorname{concat}\left(\left\{\widehat{\BF{y}}^{(\alpha)}_{\shortminus k}\colon \alpha\in I(Y,i,k)\right\}\right)\right) \label{eq:intermediate2}.
\end{align}
%%%%%%
Based on the above result, we can apply Thm.~\ref{thm:partial-derivative-hv} on the set $\{\widehat{\BF{y}}^{(\alpha)}_{\shortminus k}\colon \alpha\in I(Y,i,k)\}\subset \R^{m-1}$ for computing the second-order derivatives, leading to the following corollary.
%%%%%%
\begin{corollary}\label{corollary:off-diagonal-partial-derivatives}
Let $p = l, \text{ if } l < k; p = l-1, \text{ if } l > k$ and $\widehat{Y} = \{\widehat{\BF{y}}^{(\alpha)}_{\shortminus k}\colon \alpha\in I(Y,i,k)\}$, where $\widehat{\BF{y}}^{(\alpha)}_{\shortminus k} = \textsc{clip}(\BF{y}^{(\alpha)}_{\shortminus k};\BF{y}^{(i)}_{\shortminus k})$ and $\BF{y}^{(\alpha)}_{\shortminus k} = \proj_k(\BF{y}^{\alpha})$.
Also, we define the index set $I(\widehat{Y}, j, p) = \{\alpha\in [1..|\widehat{Y}|]\colon \widehat{y}^{(\alpha)}_p < \widehat{y}^{(j)}_{p}\}$ 
($\widehat{y}^{(j)}_{p}$ is the $p$-th component of $\widehat{\BF{y}}^{(j)}_{\shortminus k}$). For $j\neq i, l\neq k$, the partial derivative $\partial(\partial \HVY/\partial y_k^{(i)})/\partial y_l^{(j)}$ admits the following expression:
\begin{equation}
    \frac{\partial^2 \HVY}{\partial y_l^{(j)}\partial y_k^{(i)}} = -\HVC\left(\proj_{p}(\widehat{\BF{y}}^{(j)}_{\shortminus k}), \{\proj_{p}(\widehat{\BF{y}}^{(\alpha)}_{\shortminus k})\colon \alpha\in I(\widehat{Y}, j, p)\}\right), \label{eq:partial-ij-cross-term}
\end{equation}
which is $m-2$-dimensional Lebesgue measure. This recursive computation is employed by line 24 of Alg.~\ref{alg:general}. 
\end{corollary}
%%%%%%
\begin{proof}
Considering Eq.~\eqref{eq:intermediate2}, the second-order derivatives of interest can be obtained by apply Thm.~\ref{thm:partial-derivative-hv} on the set $\widehat{Y}=\{\widehat{\BF{y}}^{(\alpha)}_{\shortminus k}\colon \alpha\in I(Y,i,k)\}\subset \R^{m-1}$.
\end{proof}
%%%%%%
In summary, with Eqs.~\eqref{eq:partial-ii-diagonal-term}~\eqref{eq:partial-ii-cross-term}~\eqref{eq:partial-ij-diagonal-term}~\eqref{eq:partial-ij-cross-term}, we have specified the computation of all entries of column $ik$ of matrix $\BF{A}$. Iterating over all the columns will compute the full matrix for our needs. To calculate the hypervolume Hessian matrix, it suffices to substitute term $\partial^2 \HVY/\partial\Y \partial\Y^\top$ in Eq.~\eqref{eq:HV-Hessian} with $\BF{A}$ and evaluate the equation (lines 20 - 23 in Alg.~\ref{alg:general}).

\section{Numerical Examples}
\label{sec:numerical}
In this section, we showcase some numerical examples of the computation of the hypervolume Hessian matrix. For the sake of comprehensibility, we only compute the Hessian w.r.t.~the objective points since the Hessian matrix w.r.t.~the decision points can be easily obtained using the former (see Eq.~\eqref{eq:HV-Hessian}). We specify the objective points for the numerical problem below.
%%%%%
\begin{itemize}
    \item Example 1: $m=3$, $n=2$, $\Y=[(5, 3, 7)^\top, (2, 1, 10)^\top]^\top$, and $\r=(9, 10, 12)^\top$.
    \item Example 2: $m=3$, $n=3$, $\Y=[(8, 7, 10)^\top, (4, 11, 17)^\top, (2, 9, 21)^\top]^\top$, and $\r=(10, 13, 23)^\top$.
    \item Example 3: $m=3$, $n=6$, $\Y=[(16, 23, 1)^\top, (14, 32, 2)^\top, (12, 27, 3)^\top, (10, 21, 4)^\top, (8, 33, 5)^\top, (6.5, 31, 6)^\top]^\top$, and $\r=(17, 35, 7)^\top$.
    \item Example 4: $m=4$, $n=5$, $\Y=[(16, 23, 1, 8)^\top, (14, 32, 2, 5)^\top, (12, 27, 3, 1)^\top, (10, 21, 4, 9)^\top, (8, 33, 5, 3)^\top]^\top$, and $\r=(17, 35, 7, 10)^\top$.
\end{itemize}
%%%%%
We illustrate the corresponding Hessian matrices as heatmaps in Fig.~\ref{fig:numerical-examples}, from which we see clearly a high sparsity in all cases. Moreover, for the second example with $n=3$ points in 3-D objective space, as predicted by Theorem \ref{thm:3Dsparsity}, we obtain exactly $12n-6$ $(=30)$ positive components. Also, we have verified the above computation by comparing the results obtained from \texttt{Python}, \texttt{Mathematica}, and automatic differentiation performed in our previous work~\cite{WangEDHS22}.
%%%%%%
\begin{figure}[t]
    \centering
    \subfloat[Example 1]{\includegraphics[width=.5\textwidth, trim=10mm 5mm 20mm 16mm, clip]{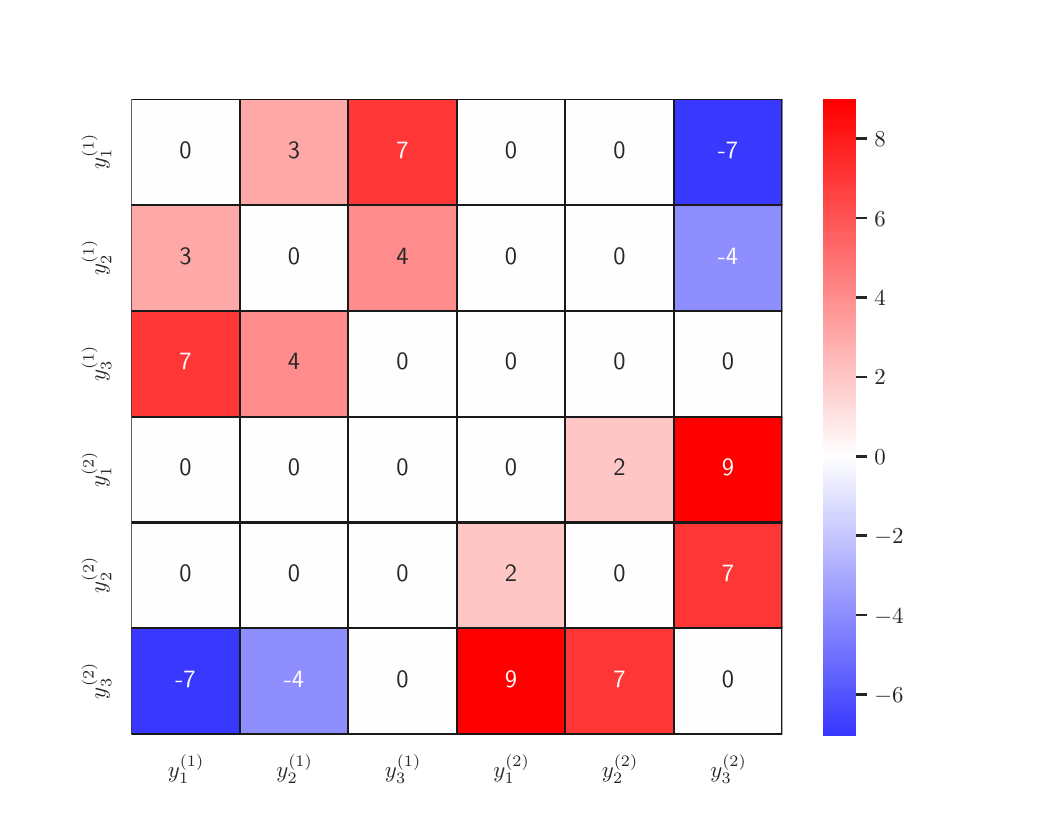}}
    \subfloat[Example 2]{\includegraphics[width=.5\textwidth, trim=10mm 5mm 20mm 16mm, clip]{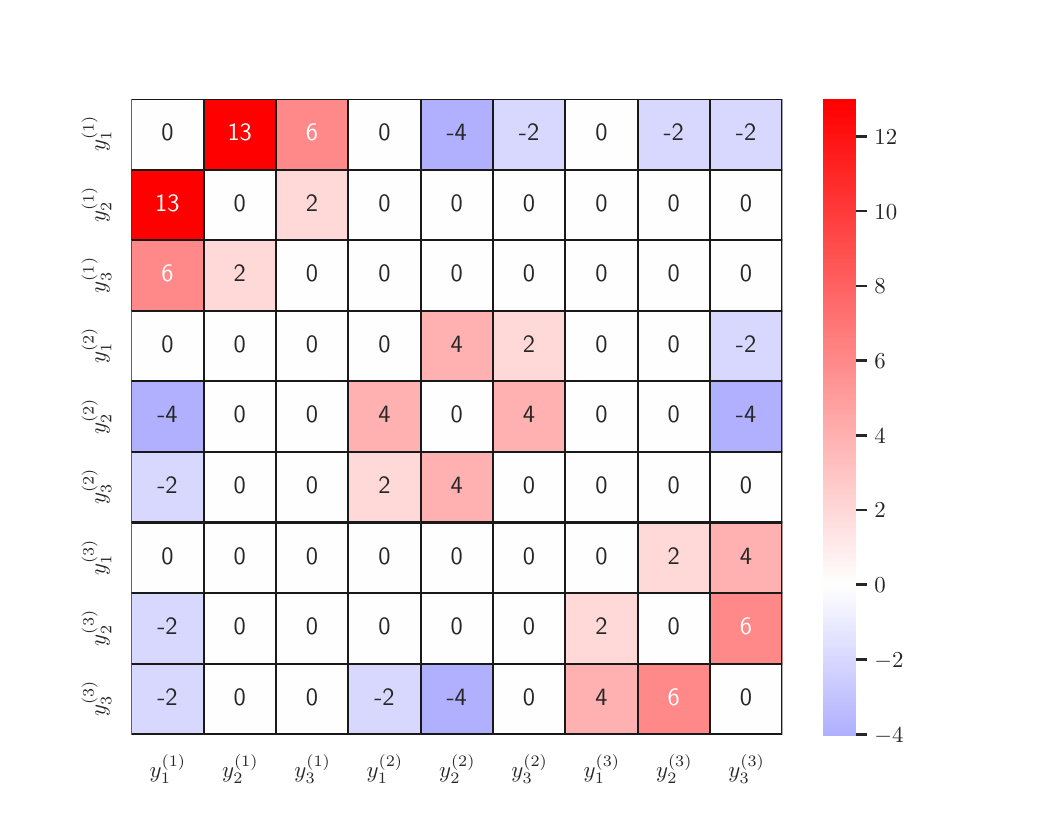}}\\
%    \vspace{-3mm}
    \subfloat[Example 3]{\hspace{-3mm}\includegraphics[width=.55\textwidth, trim=10mm 5mm 20mm 14mm, clip]{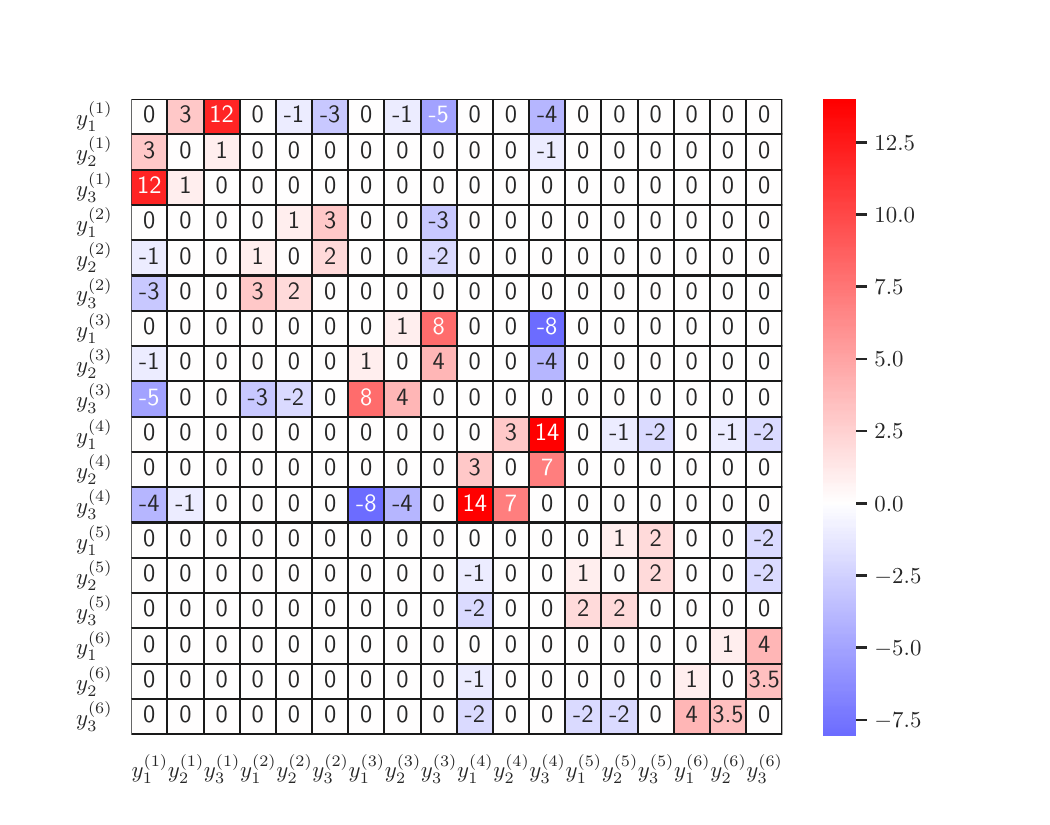}}
    \subfloat[Example 4]{\includegraphics[width=.55\textwidth, trim=3mm 3mm 10mm 3mm, clip]{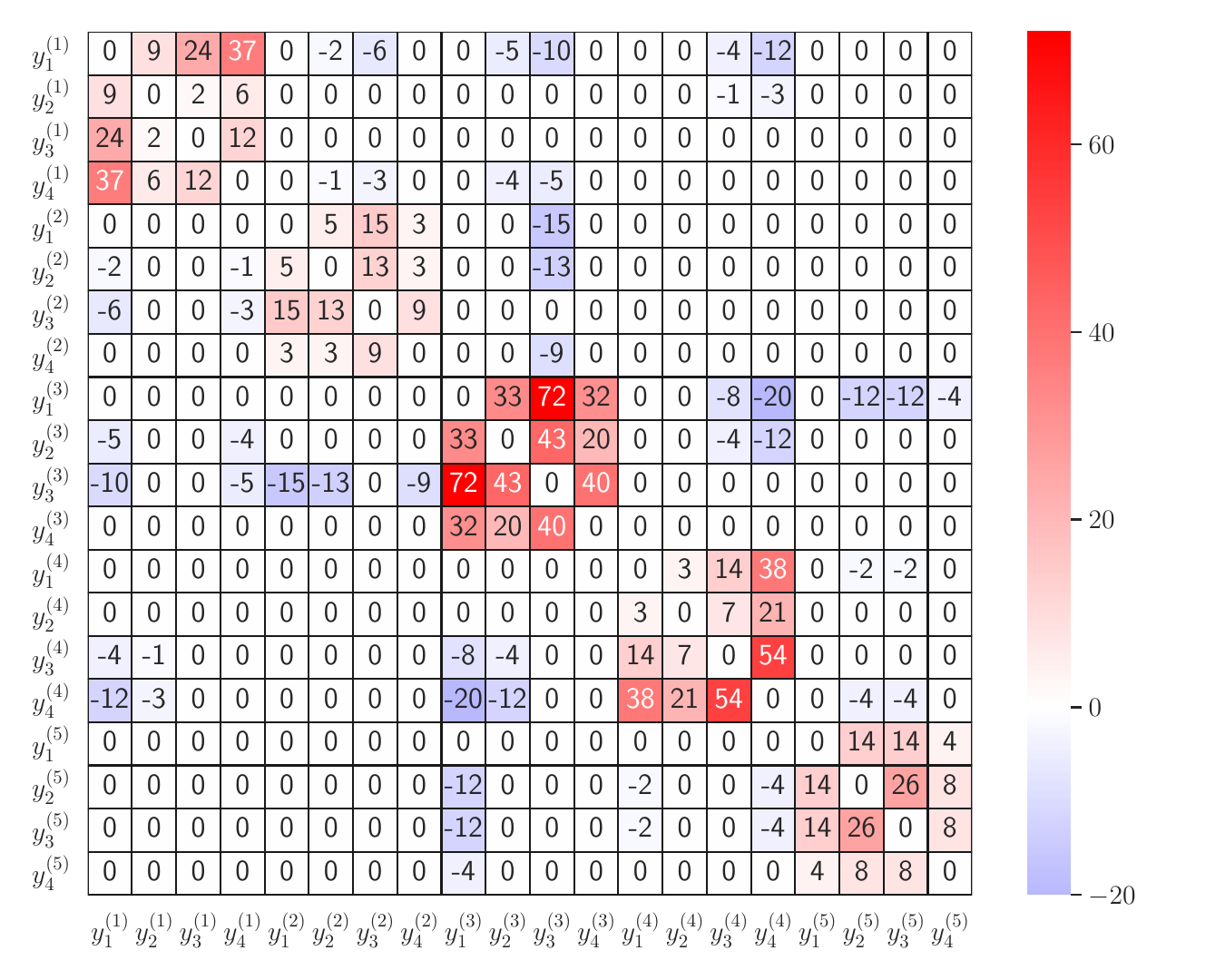}}
%    \hfill\mbox{}
    \caption{The hypervolume Hessian matrix $\partial^2 \HVY/\partial \Y\partial \Y^\top$ w.r.t.~objective vectors rendered as heatmaps for three examples given in Sec.~\ref{sec:numerical}.}
    \label{fig:numerical-examples}
\end{figure}

\section{Discussion and Outlook}
\label{sec:disuss}
This paper highlights two approaches for computing the components of the Hessian matrix of the hypervolume indicator $\HV$ of a multi-set of points in objective space and of a multi-set of points in the decision space $\HVF$.
The approach of set-scalarization, as originated in~\cite{Emmerich07, EmmerichD12} for the gradient of the hypervolume indicator. The main results of the paper are as follows:
\begin{enumerate}
\item the hypervolume indicator of $\HVF$ can now be computed analytically not only for the bi-objective case as in \cite{Sosa-HernandezS20}, but also for more than two objective functions (Theorem~\ref{thm:3Dtrisweepsoundness}). 
\item the time complexity of computing all non-zero components of the 3-D hypervolume indicator $\HV$ for vectorized sets $\mathbf{Y}$ with $n$ points  in general position is in $\Theta(n \log n)$. (Theorem~\ref{thm:3Dtime})
\item The number of non-zero components of the Hessian matrix is at most $12n-6$. The space complexity of the Hessian matrix computation and the space required to store all components is in $O(n)$. (Theorem~\ref{thm:3Dsparsity}).
\item  it holds that $\partial \HVY /\partial y_k$ is always the hypervolume contribution of the projection of $y_k$ along axis $k$. (Theorem~\ref{thm:partial-derivative-hv}).
\item the analytical computation of the higher derivatives of the $m$-dimensional hypervolume indicator $\HV$ for $m>1$can be formulated by computing the gradient of the gradient, which can be essentially achieved by the recursive application of Theorem (Theorem~\ref{thm:partial-derivative-hv}) (computing the $m-2$-dimensional projection's contributions along the $y_k$ axis, of the $m-1$ dimensional projections' hypervolume contributions along the $y_k$); and taking special care of the role of the reference points and signs of non-zero components as detailed in Alg.~\ref{alg:general}. 
\end{enumerate}
%%%%
Some interesting next steps would be to
%%%%
\begin{enumerate}
\item investigate the rank of the Hypervolume Hessian matrix and its numerical stability of second-order methods that use the Hessian matrix of $\HVF$ (or its inverse) in their iteration, such as the Hypervolume Newton Method~\cite{sosa2014hypervolume, Sosa-HernandezS20}.
\item find (asymptotically) efficient algorithms for the computation of the higher-order derivative tensors and for more than three-dimensional cases. The latter might, however, find the asymptotical time complexity of the N-D Hessian matrix computation might turn out to be a difficult endeavor, as it is not even known what the asymptotical time complexity of $\HV$ is. What is more promising is to bound the number of the non-zero components in the Hessian matrix, which is related to the number of $n-2$ dimensional facets in the ortho-convex polyhedron that marks the measured region of $\HV$. It is conjectured that in the $m$-dimensional case, it also grows linearly in $n$ (the number of points in the approximation set) but exponentially in $m$ (the number of objectives). However, dimension sweep algorithms also yield high efficiency in the 4-D case and can probably be easily adapted~\cite{guerreiro2012fast}.
\end{enumerate}
More generally, it is remarked that besides the hypervolume indicator also, other measures have been proposed for the quality of Pareto front approximations, such as the inverted generational distance~\cite{ishibuchi:15} or the averaged Hausdorff distance~\cite{schuetze_scalar:16,uribe2020set}, with sometimes advantageous properties regarding the uniformity the point distributions in their maximum. Also, for those, the approach of vectorization of the input set is promising, and the analytical computation and subsequent analysis of the Hessian matrix can be an interesting approach.
%%%%
The code for computing analytically the Hessian matrix of $\HV$ has been validated on example data and made available in a GitHub repository.\footnote{\url{https://github.com/wangronin/HypervolumeDerivatives}} The repository includes an implementation based on Alg.~\ref{alg:general} in \texttt{Python} and in \texttt{Mathematica}, as well as the data of the examples.

{\bf Remark}: The authors have been listed alphabetically in this paper, and all authors have contributed to the completion of the manuscript.\\
{\bf Author contributions}: HW + ME + AD: Concept and formulation of the general analytical expression for the matrix, ME: 3-D Dimension sweep algorithm and 3-D Complexity Analysis; AD+HW: Improved mathematical notation and numerical validation experiments; HW: theoretical analysis of the general N-dimensional Hessian matrix; All authors: General set-vectorization concepts, motivation, revision/editing of formulation of the paper. 

\bibliographystyle{alpha}
\bibliography{ref}

\end{document}